\documentclass[oneside,english]{amsart}
\usepackage[T1]{fontenc}
\usepackage[latin9]{inputenc}
\usepackage{textcomp}
\usepackage{mathrsfs}
\usepackage{amsthm}
\usepackage{amssymb}
\usepackage{esint}

\makeatletter
\numberwithin{equation}{section}
\numberwithin{figure}{section}
 \theoremstyle{definition}
 \newtheorem*{defn*}{\protect\definitionname}
\theoremstyle{plain}
\newtheorem{thm}{\protect\theoremname}
  \theoremstyle{plain}
  \newtheorem{lem}[thm]{\protect\lemmaname}
  \theoremstyle{remark}
  \newtheorem*{rem*}{\protect\remarkname}
  \theoremstyle{definition}
  \newtheorem{example}[thm]{\protect\examplename}
  \theoremstyle{remark}
  \newtheorem{rem}[thm]{\protect\remarkname}
  \theoremstyle{plain}
  \newtheorem{cor}[thm]{\protect\corollaryname}

\makeatother

\usepackage{babel}
  \providecommand{\corollaryname}{Corollary}
  \providecommand{\definitionname}{Definition}
  \providecommand{\examplename}{Example}
  \providecommand{\lemmaname}{Lemma}
  \providecommand{\remarkname}{Remark}
\providecommand{\theoremname}{Theorem}

\begin{document}

\title{Extremal Domains for Self-Commutators in the Bergman Space}

\author{Matthew Fleeman and Dmitry Khavinson}
\begin{abstract}
In \cite{Olsen}, the authors have shown that Putnam's inequality
for the norm of self-commutators can be improved by a factor of $\frac{1}{2}$
for Toeplitz operators with analytic symbol $\varphi$ acting on the
Bergman space $A^{2}(\Omega)$. This improved upper bound is sharp
when $\varphi(\Omega)$ is a disk. In this paper we show that disks
are the only domains for which the upper bound is attained.%
\thanks{The authors acknowledge support from the NSF grant DMS - 0855597.%
}
\end{abstract}
\maketitle

\section{Introduction}

Let $H$ be a complex Hilbert Space with inner product $\langle\cdot,\cdot\rangle:H\times H\rightarrow\mathbb{C}$
and $T$ be a bounded linear operator with adjoint $T^{*}$. Assume
$[T^{*},T]:=T^{*}T-TT^{*}\geq0$, i.e. $T$ is hyponormal, then Putnam's
inequality states 
\[
\Vert[T^{*},T]\Vert\leq\frac{Area(sp(T))}{\pi},
\]
where $sp(T)$ denotes the spectrum of $T$ (cf. \cite{Axler}).

In \cite{Khavinson}, it was proved that when $H$ is the Hardy space,
$H^{2}(\Omega)$, or the Smirnov space, $E^{2}(\Omega)$ (cf. \cite[p.2 and p.173]{Duren1}),
where $\Omega$ is a domain bounded by finitely many rectifiable curves,
this inequality is sharp. Indeed, if we take $T=T_{\varphi}$ to be
the Toeplitz operator with symbol $\varphi$ analytic in a neighborhood
of $\overline{\Omega}$, then by the Spectral Mapping Theorem (cf.
\cite[p.263]{Rudin}) $sp(T_{\varphi})=\overline{\varphi(\Omega)}$,
and the following lower bound holds:
\begin{equation}
\Vert[T_{\varphi}^{*},T_{\varphi}]\Vert\geq\frac{4Area^{2}(\varphi(\Omega))}{\Vert\varphi'\Vert_{E^{2}(\Omega)}^{2}\cdot P(\Omega)},\label{eq:khavinson}
\end{equation}
where $P(\Omega)$ denotes the perimeter of $\Omega$. 

Since $[T_{\varphi}^{*},T_{\varphi}]$ is a positive operator, an
interesting consequence follows from (\ref{eq:khavinson}) by setting
$\varphi(z)=z$, so that $\Vert\varphi'\Vert_{E^{2}(\Omega)}^{2}=\Vert1\Vert_{E^{2}(\Omega)}^{2}=P(\Omega)$,
and combining (\ref{eq:khavinson}) with Putnam's inequality, we obtain
\begin{equation}
P^{2}(\Omega)\geq4\pi Area(\Omega),\label{eq:IsoPerimetric}
\end{equation}
which is the classical isoperimetric inequality. The equality in (\ref{eq:IsoPerimetric})
holds if and only if $\Omega$ is a disk.

We are interested in exploring similar questions in a Bergman space
setting. Recall that the Bergman space $A^{2}(\mathbb{D})$ is defined
by:
\[
A^{2}(\mathbb{D}):=\{f\in Hol(\mathbb{D}):\int_{\mathbb{D}}\vert f(z)\vert^{2}dA(z)<\infty\},
\]
where $dA$ denotes the area measure on $\mathbb{D}$. $A^{2}(\Omega)$
is defined accordingly (cf. \cite{Duren2}). The orthogonal projection
from $L^{2}(\mathbb{D},dA)=L^{2}(\mathbb{D})$ onto $A^{2}(\mathbb{D})$
is called the Bergman projection and has integral representation
\[
Pf(z)=\int_{\mathbb{D}}\frac{f(\omega)}{(1-\bar{\omega}z)^{2}}dA(\omega).
\]

Recently, Bell, Ferguson, and Lundberg in \cite{Bell} obtained a
different lower bound for $[T_{\varphi}^{*},T_{\varphi}]$ when $T_{\varphi}$
acts on the Bergman space $A^{2}(\Omega)$. This lower bound turned
out to be connected with the torsional rigidity of $\Omega$ (cf.
\cite[p.2]{Polya}). Intuitively, if we imagine a cylindrical object
with cross-section $\Omega$, then the torsional rigidity quantifies
the resistance to twisting. There are several equivalent definitions
(cf. \cite[pp.87-89]{Polya}). The one used in \cite{Bell} and \cite{Olsen}
is the following:
\begin{defn*}
If $\Omega$ is a simply connected domain, the torsional rigidity
$\rho=\rho(\Omega)$ is 
\[
\rho=2\int_{\Omega}\nu,
\]
where $\nu$ is the unique solution to the Dirichlet problem
\[
\begin{cases}
\Delta\nu & =-2\\
\nu\vert_{\partial\Omega} & =\;0
\end{cases}.
\]

\end{defn*}
In \cite{Bell} it is shown that for $T_{z}$ acting on $A^{2}(\Omega)$,
\[
\Vert[T_{z}^{*},T_{z}]\Vert\geq\frac{\rho(\Omega)}{Area(\Omega)}.
\]
The authors also conjectured that in the Bergman space setting Putnam's
inequality could be improved by a factor of $\frac{1}{2}$. This conjecture
was recently proven by Olsen and Reguera in \cite{Olsen}. Combined
with the lower bound given by Bell, Ferguson, and Lundberg, this yields
a new proof of the St. Venant inequality 
\[
\rho(\Omega)\leq\frac{Area(\Omega)^{2}}{2\pi}.
\]

In this note we show that in the Bergman space setting disks are the
only extremal domains for which $\Vert T_{z}^{*},T_{z}\Vert$ achieves
it's maximal bound of $\frac{Area(\Omega)}{2\pi}$. More precisely,
in section 2, we present a simple argument illustrating that the upper
bound in Putnam's inequality can only be attained when $\Omega$ is
a disk in any Hilbert space, while in the Bergman space setting $\Vert T_{\varphi}^{*},T_{\varphi}\Vert=\frac{Area(\varphi(\Omega))}{2\pi}=\frac{1}{2}$
when $\varphi(\Omega)=\mathbb{D}$, the unit disk. In section 3, we
give a sketch of Olsen and Reguera's proof of the improved upper bound
in the Bergman setting. This is needed for our argument in section
4 where we show that the upper bound for $\Vert T_{\varphi}^{*},T_{\varphi}\Vert$
is achieved if and only if $\varphi(\Omega)$ is a disk. This gives
another proof of the well known fact that St. Venant's inequality
becomes equality only for disks.

\section{Non-sharpness of Putnam's Inequality}

In this section, we illustrate why Putnam's inequality is not sharp
in a Bergman space setting. We start with the following elementary
Lemma found in \cite[p.13]{Duren2}.
\begin{lem}
Suppose $\omega=\varphi(z)$ maps a domain $D$ conformally onto a
domain $\Omega$. Then the linear map $T(f)=g$ defined by 
\[
g(z)=f(\varphi(z))\varphi'(z)
\]
defines an isometry of $A^{2}(\Omega)$ onto $A^{2}(D)$.\end{lem}
\begin{proof}
That $T$ is an isometry is clear from the fact that 
\[
\int_{\Omega}\vert f(w)\vert^{2}dA(\omega)=\int_{D}\vert f(\varphi(z))\vert^{2}\vert\varphi'(z)\vert^{2}dA(z),
\]
where $\vert\varphi'(z)\vert^{2}$ is the Jacobian of the conformal
map $\varphi.$

To see that $T$ is onto, let $g\in A^{2}(D)$ and let $z=\psi(w)$
be the inverse mapping. Then $f(\omega)=g(\psi(\omega))\psi'(\omega)$
is in $A^{2}(\Omega)$ and $T(f)=g$ since
\[
T(f)=f(\varphi(z))\varphi'(z)=g(\psi(\varphi(z))\varphi'(\psi(\omega))\psi'(\omega),
\]
and we can write $\psi'(\omega)=\frac{1}{\varphi'(\psi(\omega))}$,
which is well defined on $D$ because $\varphi'\vert_{D}\neq0$. So
$T(f)=g$ and $T$ is onto as claimed.
\end{proof}
The following statement is now straightforward.
\begin{thm}
Suppose $\Omega$ is a bounded Jordan domain and $\varphi:\mathbb{D}\rightarrow\Omega$
is a conformal mapping. Then 
\[
\Vert[T_{\varphi}^{*},T_{\varphi}]\Vert_{A^{2}(\mathbb{D})\rightarrow A^{2}(\mathbb{D})}=\Vert[T_{z}^{*},T_{z}]\Vert_{A^{2}(\Omega)\rightarrow A^{2}(\Omega)}.
\]
\end{thm}
\begin{proof}
We start with the following straightforward calculation (cf. \cite{Axler}).
If we take $A_{1}^{2}(\mathbb{D})$ to be the unit ball of $A^{2}(\mathbb{D})$,
we have that
\[
\Vert[T_{\varphi}^{*},T_{\varphi}]\Vert=\sup_{f\in A_{1}^{2}(\mathbb{D})}\langle[T_{\varphi}^{*},T_{\varphi}]f,f\rangle
\]
\[
=\sup_{f\in A_{1}^{2}(\mathbb{D})}\left(\Vert T_{\varphi}f\Vert_{A^{2}(\mathbb{D})}^{2}-\Vert T_{\varphi}^{*}f\Vert_{A^{2}(\mathbb{D})}^{2}\right)
\]
\[
=\sup_{f\in A_{1}^{2}(\mathbb{D})}\left(\Vert\varphi f\Vert_{A^{2}(\mathbb{D})}^{2}-\Vert P(\bar{\varphi}f)\Vert_{A^{2}(\mathbb{D})}^{2}\right)
\]
\[
=\sup_{f\in A_{1}^{2}(\mathbb{D})}\left(\Vert\varphi f\Vert_{L^{2}(\mathbb{D})}^{2}-\Vert P(\bar{\varphi}f)\Vert_{A^{2}(\mathbb{D})}^{2}\right).
\]
Thus we have that 
\[
\Vert[T_{\varphi}^{*},T_{\varphi}]\Vert=\sup_{g\in A_{1}^{2}(\mathbb{D})}\left(\Vert\varphi f\Vert_{L^{2}(\mathbb{D})}^{2}-\Vert P(\bar{\varphi}g)\Vert_{A^{2}(\mathbb{D})}^{2}\right)
\]
\[
=\sup_{g\in A_{1}^{2}(\mathbb{D})}\{\inf_{f\in A^{2}(\mathbb{D})}\{\Vert\bar{\varphi}g-f\Vert_{L^{2}(\mathbb{D})}^{2}\}\}.
\]
Fixing $f,g\in A^{2}(\mathbb{D})$, with $g\in A_{1}^{2}(\mathbb{D})$,
and letting $\psi=\varphi^{-1}$, we see that 
\[
\Vert\bar{\varphi}g-f\Vert_{L^{2}(\mathbb{D})}^{2}=\int_{\mathbb{D}}\vert\bar{\varphi}g-f\vert^{2}dA
\]
\[
=\int_{\Omega}\vert\bar{z}g(\psi(\omega))-f(\psi(\omega))\vert^{2}\vert\psi'(\omega)\vert^{2}dA(\omega)
\]
\[
=\int_{\Omega}\vert\bar{z}g(\psi(\omega))\psi'(\omega)-f(\psi(\omega))\psi'(\omega)\vert^{2}dA(\omega).
\]

By Lemma 1, $T(f)=f(\psi(\omega))\psi'(\omega)$ is a surjective isometry
from $A^{2}(\mathbb{D})$ onto $A^{2}(\Omega)$. So, we have that
\[
\sup_{g\in A_{1}^{2}(\mathbb{D})}\{\inf_{f\in A^{2}(\mathbb{D})}\{\Vert\bar{\varphi}g-f\Vert_{L^{2}(\mathbb{D})}^{2}\}\}=\sup_{g\in A_{1}^{2}(\Omega)}\{\inf_{f\in A^{2}(\Omega)}\{\Vert\bar{z}g-f\Vert_{L^{2}(\Omega)}^{2}\}\},
\]
and the proof is complete. 
\end{proof}
This leads to the following interesting observation.
\begin{thm}
\label{thm:PutEqDisks}Let $\varphi$ and $\Omega$ be as in Theorem
2. Then $\Vert[T_{\varphi}^{*},T_{\varphi}]\Vert$ can only achieve
the upper bound stated in Putnam's inequality (cf. \cite{Putnam})
if $\varphi(\mathbb{D})$ is a disk.\end{thm}
\begin{proof}
We argue as follows. Let $A_{1}^{2}(\Omega)$ be the unit ball in
$A^{2}(\Omega)$. By Theorem 2, we have that
\[
\Vert[T_{\varphi}^{*},T_{\varphi}]\Vert_{A^{2}(\mathbb{D})\rightarrow A^{2}(\mathbb{D})}=\Vert[T_{z}^{*},T_{z}]\Vert_{A^{2}(\Omega)\rightarrow A^{2}(\Omega)}=\sup_{g\in A_{1}^{2}(\Omega)}\{\inf_{f\in A^{2}(\Omega)}\{\Vert\overline{z}g-f\Vert_{A^{2}(\Omega)}^{2}\}\}.
\]
Fix $g\in A_{1}^{2}(\Omega)$, we have 
\[
\inf_{f\in A^{2}(\Omega)}\Vert\bar{z}g-f\Vert_{L^{2}(\Omega)}^{2}=\inf_{f\in A^{2}(\Omega)}\int_{\Omega}\vert\bar{z}g-f\vert^{2}dA
\]
 
\[
\leq\inf_{h:\, gh\in A^{2}(\Omega)}\int_{\Omega}\vert\bar{z}-f\vert^{2}\vert g\vert^{2}dA\leq\inf_{h:\, gh\in A^{2}(\Omega)}\Vert\bar{z}-h\Vert_{\infty}^{2}
\]
since $g\in A_{1}^{2}(\Omega)$. Further, since the polynomials $\mathcal{P}$
are dense in $H^{\infty}(\Omega)$ for any bounded Jordan domain $\Omega$,
and since for all $g\in A_{1}^{2}(\Omega)$, and all $p\in\mathcal{P}$,
we have that $gp\in A^{2}(\Omega)$, we obtain from the last inequality
that 
\[
\inf_{f\in A^{2}(\Omega)}\Vert\bar{z}g-f\Vert_{L^{2}(\Omega)}^{2}\leq\inf_{h\in R(\overline{\Omega})}\Vert\bar{z}-h\Vert_{L^{\infty}(\Omega)}^{2}
\]
where $R(\overline{\Omega})$ is the uniform closure of the algebra
of rational functions in $\Omega$ with poles outside $\overline{\Omega}$.
In \cite{Alexander}, Alexander proved that 
\[
\inf_{f\in R(\overline{\Omega})}\Vert\bar{z}-f\Vert_{L^{\infty}(\Omega)}\leq\sqrt{\frac{Area(\Omega)}{\pi}},
\]
and further that equality is achieved if, and only if, $\Omega$ is
a disk (cf. \cite{Axler,Gamelin}). The theorem now immediately follows.\end{proof}
\begin{rem*}
If we take $H$ to be any Hilbert space and $T$ to be any subnormal
operator with a rationally cyclic vector, then there is a positive
finite Borel measure $\mu$ on $sp(T)$ such that $T$ is unitarily
equivalent to multiplication by $z$ on $R^{2}(sp(T),\,\mu)$ which
is the closure of $R(sp(T))$ in $L^{2}(sp(T),\,\mu)$ (cf. \cite{Axler}).
From this, repeating the above argument word for word, we obtain that
if 
\[
\Vert[T^{*},T]\Vert=\frac{Area(sp(T))}{\pi},
\]
then $sp(T)$ must be a disk. The case when $T$ does not have a rationally
cyclic vector follows from the above case as in \cite{Axler}, so
that the above theorem extends to all Hilbert spaces and any subnormal
operator $T$. 
\end{rem*}
The following example shows that the converse fails, and in particular
fails for Bergman spaces.
\begin{example}
Let $\varphi(z)=z^{k}$ for some $k\in\mathbb{N}$, and let $T_{\varphi}:A^{2}(\mathbb{D})\rightarrow A^{2}(\mathbb{D})$,
and recall that $P:L^{2}(\mathbb{D})\rightarrow A^{2}(\mathbb{D})$
is the orthogonal projection of $L^{2}(\mathbb{D})$ onto $A^{2}(\mathbb{D})$.
As we showed in Theorem 2, 
\[
\Vert[T_{\varphi}^{*},T_{\varphi}]\Vert=\sup_{g\in A_{1}^{2}(\mathbb{D})}\left(\Vert\varphi g\Vert_{L^{2}(\mathbb{D})}^{2}-\Vert P(\bar{\varphi}g)\Vert_{A^{2}(\mathbb{D})}^{2}\right).
\]
Let $\psi_{n}(z)=(\frac{n+1}{\pi})^{\frac{1}{2}}z^{n}$, where $n=0,1,2,\ldots$.
The collection $\{\psi_{n}(z)\}_{n=0}^{\infty}$ forms an orthonormal
basis for $A^{2}(\mathbb{D})$ (cf. \cite[p. 11]{Duren2}). For $g\in A_{1}^{2}(\mathbb{D}),$
we may write
\[
g(z)=\sum_{n=0}^{\infty}\hat{g}(n)\psi_{n}(z),
\]
where $\hat{g}(n):=\langle g,\psi_{n}\rangle$ and $\sum_{n=0}^{\infty}\vert\hat{g}(n)\vert^{2}=1$.
Since we have an orthonormal basis at hand, we may calculate $P(\bar{\varphi}g)$
explicitly. 
\[
P(\bar{z}^{k}g)=\sum_{n=0}^{\infty}\langle\bar{z}^{k}g,\psi_{n}\rangle\psi_{n}.
\]
Calculating $\langle\bar{z}^{k}g,\psi_{n}\rangle$, we find that 
\[
\langle\bar{z}^{k}g,\psi_{n}\rangle=\langle\bar{z}^{k}\sum_{m=0}^{\infty}\hat{g}(m)\psi_{m},\psi_{n}\rangle,
\]
where 
\[
\langle\bar{z}^{k}\hat{g}(m)\psi_{m},\psi_{n}\rangle=\int_{\mathbb{D}}\frac{\sqrt{(m+1)(n+1)}}{\pi}\hat{g}(m)z^{m}\bar{z}^{n+k}dA
\]
\[
=\frac{\sqrt{(m+1)(n+1)}}{\pi}\hat{g}(m)\frac{2\pi}{m+n+k+2}\delta_{m,n+k}.
\]
Where $\delta_{i,j}$ is the Kronecker symbol. Thus, 
\[
\langle\bar{z}^{k}g,\psi_{n}\rangle=\left(\frac{n+1}{n+k+1}\right)^{\frac{1}{2}}\hat{g}(n+k),
\]
and so we obtain that 
\begin{equation}
\Vert P(\bar{z}^{k}g)\Vert_{A^{2}(\mathbb{D})}^{2}=\sum_{n=0}^{\infty}\frac{n+1}{n+k+1}\vert\hat{g}(n+k)\vert^{2}.\label{eq:example1}
\end{equation}
Similarly, when we calculate $\Vert z^{k}g\Vert_{A^{2}(\mathbb{D})}^{2}$,
we find that
\[
\langle z^{k}\hat{g}(m)\psi_{m},\psi_{n}\rangle=\int_{\mathbb{D}}\frac{\sqrt{(m+1)(n+1)}}{\pi}\hat{g}(m)z^{m+k}\bar{z}^{n}dA
\]
\[
=\frac{\sqrt{(m+1)(n+1)}}{\pi}\hat{g}(m)\frac{2\pi}{m+n+k+2}\delta_{m+k,n}.
\]
Thus
\[
\langle z^{k}g,\psi_{n}\rangle=\begin{cases}
\sqrt{\frac{n-k+1}{n+1}}\hat{g}(n-k) & n\geq k\\
0 & n<k
\end{cases}.
\]
Hence, 
\begin{equation}
\Vert z^{k}g\Vert_{L^{2}(\mathbb{D})}^{2}=\sum_{n=k}^{\infty}\frac{n-k+1}{n+1}\vert\hat{g}(n-k)\vert^{2}.\label{eq:example2}
\end{equation}
Combining (\ref{eq:example1}) and (\ref{eq:example2}), we obtain
that
\[
\Vert[T_{\varphi}^{*},T_{\varphi}]\Vert=\sup_{g\in A_{1}^{2}(\mathbb{D})}\{\sum_{n=k}^{\infty}\frac{n-k+1}{n+1}\vert\hat{g}(n-k)\vert^{2}-\sum_{n=0}^{\infty}\frac{n+1}{n+k+1}\vert\hat{g}(n+k)\vert^{2}\}
\]
\[
\sup_{g\in A_{1}^{2}(\mathbb{D})}\{\sum_{n=0}^{k-1}\frac{n+1}{n+k+1}\vert\hat{g}(n)\vert^{2}+\sum_{n=k}^{\infty}(\frac{n+1}{n+k+1}-\frac{n-k+1}{n+1})\vert\hat{g}(n)\vert^{2}\}
\]
\[
\leq\sup_{g\in A_{1}^{2}(\mathbb{D})}\{\sum_{n=0}^{k-1}\frac{n+1}{n+k+1}\vert\hat{g}(n)\vert^{2}+\sum_{n=k}^{\infty}\frac{k}{n+k+1}\vert\hat{g}(n)\vert^{2}\}
\]
since
\[
\frac{n+1}{n+k+1}-\frac{n-k+1}{n+1}\leq\frac{n+1}{n+k+1}-\frac{n-k+1}{n+k+1}=\frac{k}{n+k+1},\qquad k\geq0.
\]
Further, since$\frac{n+1}{n+k+1}\leq\frac{k}{2k}$ for $0\leq n\leq k-1$,
we obtain that 
\[
\Vert[T_{\varphi}^{*},T_{\varphi}]\Vert\leq\sup_{g\in A_{1}^{2}(\mathbb{D})}\frac{k}{2k}\sum_{n=0}^{\infty}\vert\hat{g}(n)\vert^{2}=\frac{1}{2}.
\]
This upper bound is achieved if we take $g=\psi_{k-1}$, so that $\Vert[T_{\varphi}^{*},T_{\varphi}]\Vert=\frac{1}{2}$,
whenever $\varphi(z)=z^{k}$ for any $k\in\mathbb{N}$. Thus, we see
that the converse to Theorem \ref{thm:PutEqDisks} fails. 

This calculation, independently done by T. Ferguson, leads to the
conjecture, following Bell et. al. that in the Bergman space setting,
Putnam's inequality can be improved by a factor of $\frac{1}{2}$.
This conjecture was recently proven by Olsen and Reguera in \cite{Olsen}.
In the following section we give a sketch of their proof, which will
be needed in §4.
\end{example}

\section{Sketch of the Olsen and Reguera proof}

In their paper, Olsen and Reguera worked with the Hankel operator
on $A^{2}(\mathbb{D})$ with symbol $\varphi\in L^{2}(\mathbb{D})$
defined by 
\[
H_{\varphi}(f):=(I-P)(\varphi f),\qquad f\in A^{2}(\mathbb{D}).
\]
They then proved the following theorem.
\begin{thm}
\label{thmOlsen-Reguera}Let $\varphi\in A^{2}(\mathbb{D})$ be in
the Dirichlet space $\mathscr{D}$, i.e $\varphi'\in A^{2}(\mathbb{D})$.
Then 
\[
\Vert H_{\bar{\varphi}}\Vert\leq\frac{1}{\sqrt{2}}\Vert\varphi'\Vert_{A^{2}(\mathbb{D})}.
\]
\end{thm}
\begin{proof}
For the reader's convenience, we give here a sketch of their proof.
For full details, cf. \cite[§2]{Olsen}. For $f\in A^{2}(\mathbb{D}),$we
write $f(z)=\sum_{n\geq0}a_{n}z^{n}$, and without loss of generality
we assume that $\Vert f\Vert_{A^{2}(\mathbb{D})}=1$, and set $\varphi(z)=\sum_{k\geq1}c_{k}z^{k}$
(we can also assume without loss of generality that $\varphi(0)=0$).
The basic strategy is to calculate $H_{\bar{\varphi}}f$ in terms
of these Taylor coefficients and obtain the desired norm estimate
by working directly with the coefficients. Crucial to our purposes
is the fact that the only inequality used in \cite{Olsen} is the
arithmetic-geometric inequality $ab\leq\frac{a^{2}+b^{2}}{2}$.

First, by computing $P(\bar{\varphi}z^{n})$ for each $n$, we find
that 
\[
H_{\bar{\varphi}}f=\bar{\varphi}(z)f(z)-P(\overline{\varphi}f)(z)
\]
\begin{equation}
=\sum_{l\geq0}\sum_{n\geq0}\overline{c_{l}}a_{n}\overline{z}^{l}z^{n}-\sum_{n\geq1}\sum_{k=0}^{n-1}\frac{k+1}{n+1}a_{n}\overline{c_{n-k}}z^{k}.\label{eq:OperatorEqn}
\end{equation}
Then, after rewriting the above expression to take advantage of the
orthogonality, we let $z=re^{i\theta}$ and integrate the modulus
squared with respect to $\frac{d\theta}{\pi}$. This yields that $\Vert H_{\bar{\varphi}}f\Vert_{A^{2}(\mathbb{D})}^{2}$
is equal to
\begin{equation}
2\sum_{k\geq1}r^{2k}\vert\sum_{n\geq0}a_{n}\overline{c_{n+k}}r^{2n}\vert^{2}+2\sum_{k\geq0}r^{2k}\vert\sum_{n\geq k+1}a_{n}\overline{c_{n-k}}(r^{2(n-k)}-\frac{k+1}{n+1})\vert^{2}.\label{eq:Norm of H}
\end{equation}
This expression is once again rewritten and then integrated with respect
to $rdr$. If we set
\[
(I):=\sum_{n,m\geq1,k\geq0}\frac{a_{n}\overline{a_{m}}c_{k+m}\overline{c_{k+n}}}{n+m+k+1},
\]
\[
(II):=\sum_{k\geq0}\sum_{n,m\geq k+1}\frac{a_{n}\overline{a_{m}}c_{m-k}\overline{c_{n-k}}(m-k)(n-k)}{(n+1)(m+1)(n+m-k+1)},
\]
then we obtain that
\[
\Vert H_{\bar{\varphi}}f\Vert_{A^{2}(\mathbb{D})}^{2}=(I)+(II).
\]
Relabeling the indices slightly, and setting $a_{n}=b_{n+1}(n+1)$,
we find that
\[
(I)=\sum_{n,m\geq1,k\geq0}b_{n}\overline{b_{m}}c_{k+m}\overline{c_{k+n}}\frac{nm}{n+m+k},
\]
\[
(II)=\sum_{n,m,k\geq1}b_{n+k}\overline{b_{m+k}}c_{m}\overline{c_{n}}\frac{mn}{n+m+k}.
\]
Using the symmetry in $m$ and $n$ we may interpret each term as
being half that of its real part so that the inequality $2Re(ab)\leq\vert a\vert^{2}+\vert b\vert^{2}$
may be applied to each term of the above expressions, and this is
the only place where inequalities occur, which yields
\[
(I)\leq\sum_{n,m\geq1,k\geq0}\left(\vert b_{n}c_{k+m}\vert^{2}+\vert b_{m}c_{k+n}\vert^{2}\right)\frac{nm}{2(n+m+k)}=\sum_{n,m\geq1,k\geq0}\vert b_{n}c_{k+m}\vert^{2}\frac{nm}{n+m+k}=:(I_{*}),
\]
\[
(II)\leq\sum_{n,m,k\geq1}\left(\vert b_{n+k}c_{m}\vert^{2}+\vert b_{m+k}c_{n}\vert^{2}\right)\frac{nm}{2(n+m+k)}=\sum_{n,m,k\geq1}\vert b_{n+k}c_{m}\vert^{2}\frac{mn}{m+n+k}=:(II_{*}).
\]
By changing the order of summation properly with the goal of isolating
unique pairs of indices, we arrive at the expression
\[
(I_{*})+(II_{*})=\sum_{n,m\geq1}\vert b_{n}\vert^{2}\vert c_{m}\vert^{2}\frac{nm}{2}.
\]
Finally, replacing $a_{n}=b_{n+1}(n+1)$, we now see that the right
hand side exactly equals
\[
\frac{1}{2}\sum_{n,m\geq0}\vert b_{n}\vert^{2}\vert c_{m}\vert^{2}mn=\frac{1}{2}\left(\sum_{n\geq0}\frac{|a_{n}\vert^{2}}{n+1}\right)\left(\sum_{m\geq1}\vert c_{m}\vert^{2}m\right)=\frac{1}{2}\Vert f\Vert_{A^{2}(\mathbb{D})}^{2}\Vert\varphi'\Vert_{A^{2}(\mathbb{D})}^{2}.
\]
which was to be shown.\end{proof}
\begin{rem}
From here, the inequality 
\begin{equation}
\Vert[T_{\varphi}^{*},T_{\varphi}]\Vert\leq\frac{\Vert\varphi'\Vert_{A^{2}(\Omega)}^{2}}{2}\label{eq:Improved Putnam}
\end{equation}
is seen as a corollary by showing that if $\psi$ is the conformal
map from $\Omega$ to $\mathbb{D}$, then 
\[
\Vert[T_{\varphi}^{*},T_{\varphi}]\Vert_{A^{2}(\Omega)\rightarrow A^{2}(\Omega)}=\Vert H_{\bar{\varphi}}\Vert_{A^{2}(\Omega)\rightarrow L^{2}(\Omega)}^{2}=\Vert H_{\overline{\varphi\circ\psi}}\Vert_{A^{2}(\mathbb{D})\rightarrow L^{2}(\mathbb{D})}^{2},
\]
and thus we can apply Theorem 5 and the result follows. Again, refer
to \cite{Olsen} for more details. Taking $\varphi(z)=z$, and combining
(\ref{eq:Improved Putnam}) with the result of Bell, Ferguson, and
Lundberg, one arrives at a proof of the sharp St. Venant inequality
\begin{equation}
\rho(\Omega)\leq\frac{Area^{2}(\Omega)}{2\pi}.\label{eq:deStVenant}
\end{equation}
It should be noted that when $\varphi=z^{k}$ many of the terms in
(\ref{eq:Norm of H}) become zero resulting in the value we found
in Example 4 of $\frac{1}{2}$ rather than the Olsen-Reguera upper
bound of $\frac{k}{2}$. 
\end{rem}

\section{Unique extremality of the Disk}

We now show that from the proof of Theorem \ref{thmOlsen-Reguera},
we may deduce that equality is obtained in (\ref{eq:Improved Putnam})
and only if $\varphi(\Omega)$ is a disk. This will come as a corollary
to the following Theorem.
\begin{thm}
Suppose $\varphi(z)$ is analytic in $\mathbb{D}$ such that $\varphi(z)\in\mathscr{D}$,
the Dirichlet space. Further suppose that 
\[
\Vert[T_{\varphi}^{*},T_{\varphi}]\Vert_{A^{2}(\mathbb{D})\rightarrow A^{2}(\mathbb{D})}=\frac{\Vert\varphi'\Vert_{A^{2}(\Omega)}^{2}}{2}.
\]
Then $\varphi(\mathbb{D})$ is a disk. \end{thm}
\begin{proof}
Since $\varphi\in\mathscr{D},$ $H_{\overline{\varphi}}$ is compact
(cf.\cite[§7.4, p.145]{Zhu}), and so attains its norm on $A_{1}^{2}(\mathbb{D})$.
Recall from the proof of Theorem 2 that 
\[
\Vert[T_{\varphi}^{*},T_{\varphi}]\Vert_{A^{2}(\mathbb{D})\rightarrow A^{2}(\mathbb{D})}=\left(\sup_{f\in A_{1}^{2}(\mathbb{D})}\Vert\varphi f\Vert_{L^{2}(\mathbb{D})}^{2}-\Vert P(\bar{\varphi}f)\Vert_{A^{2}(\mathbb{D})}^{2}\right)
\]
 
\[
=\Vert H_{\bar{\varphi}}\Vert_{A^{2}(\mathbb{D})\rightarrow L^{2}(\mathbb{D})}^{2}.
\]
We now examine the proof of Theorem \ref{thmOlsen-Reguera} to find
exactly when equality may happen. Recall that if $f\in A_{1}^{2}(\mathbb{D})$,
then 
\[
\Vert H_{\bar{\varphi}}f\Vert_{A^{2}(\mathbb{D})}^{2}=(I)+(II)\leq(I_{*})+(II_{*})=\frac{1}{2}\Vert f\Vert_{A^{2}(\mathbb{D})}^{2}\Vert\varphi'\Vert_{A^{2}(\mathbb{D})}^{2},
\]
where $(I),$ $(II)$, $(I_{*})$, and $(II_{*})$ are as in Theorem
\ref{thmOlsen-Reguera}. The only inequality at work here is $2Re(ab)\leq\vert a\vert^{2}+\vert b\vert^{2}$,
where equality is achieved if, and only if, $a=\bar{b}.$ Thus we
find that equality is achieved if $(I)=(I_{*})$ and $(II)=(II_{*})$,
which will only happen if the following infinite system of equations
is satisfied:
\begin{equation}
b_{i}c_{j+k}=b_{j}c_{i+k}\qquad i,j\geq1,k\geq0,\label{eq:1}
\end{equation}
\begin{equation}
b_{i+k}c_{j}=b_{j+k}c_{k}\qquad i,j,k\geq1,\label{eq:2}
\end{equation}
where $\varphi(z)=\sum_{k\geq1}c_{k}z^{k}$ is given and $f(z)=\sum_{n\geq1}nb_{n}z^{n-1}$
is an extremal function in $A_{1}^{2}(\mathbb{D})$ such that the
above equations are satisfied. 

It is clear that if $c_{k}=0$ for all but a single $k$, that is
if $\varphi(z)=cz^{k}$ then the above equations can be satisfied
by a non-zero $f\in A_{1}^{2}(\mathbb{D})$. In fact, we know from
Example 4 that if we take $f=\psi_{k-1}$, then (\ref{eq:1}) and
(\ref{eq:2}) will be trivially satisfied. As we remarked above, in
this case the formula (\ref{eq:Norm of H}) is oversimplified, so
that the resulting norm is $\frac{c^{2}}{2}$ instead of our expected
upper bound of $\frac{c^{2}k}{2}$ It is also clear that the above
equations are satisfied when $\varphi(z)=\sum_{k\geq1}r^{k}z^{k}$
for some $r<1$. Here, the extremal $f=\frac{1}{\Vert\varphi\Vert_{A^{2}(\mathbb{D})}}\sum_{k\geq0}r^{k}z^{k}$.
In both cases $\varphi(\mathbb{D})$ is a disk. 

We will now show that for all other $\varphi$, (\ref{eq:1}) and
(\ref{eq:2}) only hold for $f\equiv0$. We will do this by looking
at two cases. 

First suppose that $\varphi(z)$ has at least two non-zero Taylor
coefficients, $c_{m}$, $c_{n}$, with $m<n$, and at least one zero
coefficient $c_{k_{0}}$ such that $k_{0}>n$. This encompasses all
Taylor series which do not have an infinite non-zero tail. Without
loss of generality we can assume that $k_{0}=n+1$ by taking $c_{k_{0}}$
to be the first zero coefficient after at least two non-zero coefficients.
We now assume that we have found an $f\in A_{1}^{2}(\mathbb{D})$
whose Taylor coefficients satisfy (\ref{eq:1}) and (\ref{eq:2}).
By (\ref{eq:2}), we have that 
\begin{equation}
b_{n+k}c_{m}=b_{m+k}c_{n}\qquad k\geq1,\label{eq:3}
\end{equation}
\begin{equation}
b_{n+k+1}c_{m}=b_{m+k}c_{n+k+1}\qquad k\geq1.\label{eq:4}
\end{equation}
Hence, we can conclude that $b_{j}=0$ for all $j\geq n+2$ by (\ref{eq:4}),
which implies that $b_{m+k}=0$ for all $k\geq2$ by (\ref{eq:3}).
We now let $i=m+1,\: j=m$ and choose $k$ such that $m+k=n$. Then
by (\ref{eq:1}) we have that 
\[
b_{m+1}c_{m+k}=b_{m+1}c_{n}=b_{m}c_{m+1+k}=b_{m}c_{n+1}=0,
\]
which, shows that $b_{m+1}=0$.

Now choosing $i<m+1$, $j=m+1$ and choosing $k$ such that $m+1+k=n$,
then by (\ref{eq:1}) we have that 
\[
b_{i}c_{m+1+k}=b_{i}c_{n}=b_{m+1}c_{n+k}=0.
\]
Hence, we have that in fact $b_{i}=0$ for all $i\geq1,$ which means
that $f\equiv0$.

Suppose now instead, that $\varphi(z)$ is such that its Taylor series
does have an infinite non-zero tail, but the coefficients do not exhibit
a geometric progression. This means that we can find three non-zero
coefficients, $c_{m}$, $c_{m+1}$, and $c_{m+2}$ such that
\begin{equation}
\frac{c_{m}}{c_{m+1}}\neq\frac{c_{m+1}}{c_{m+2}}.\label{eq:5}
\end{equation}
By (\ref{eq:2}), we have that
\begin{equation}
b_{m+k}c_{m+1}=b_{m+1+k}c_{m}\qquad k\geq1,\label{eq:6}
\end{equation}
\begin{equation}
b_{m+1+k}c_{m+2}=b_{m+2+k}c_{m+1}\qquad k\geq1.\label{eq:7}
\end{equation}
In particular, choosing $k=2$ in (\ref{eq:6}) and $k=1$ in (\ref{eq:7})
we have that 
\[
b_{m+2}c_{m+1}=b_{m+3}c_{m},
\]
 and 
\[
b_{m+2}c_{m+2}=b_{m+3}c_{m+1},
\]
which by (\ref{eq:5}) means that $b_{m+2}=b_{m+3}=0$. In fact, the
same argument shows that $b_{j}=0$ for all $j\geq m+2$. But then
of course, by (\ref{eq:6}) we immediately get that $b_{j}=0$ for
all $j\geq m+1$. Now once again simply let $i<m+1,$ $j=m+1$, and
$k=1$, and then by (\ref{eq:1}) we once again have that $b_{i}=0$
for all $i\geq1$, and so $f\equiv0.$ 
\end{proof}
Our result now follows as a corollary.
\begin{cor}
\label{cor:StVenantEquality}$\Vert[T_{z}^{*},T_{z}]\Vert_{A^{2}(\Omega)\rightarrow A^{2}(\Omega)}=\frac{Area(\Omega)}{2\pi}$
if, and only if, $\Omega$ is a disk. \end{cor}
\begin{proof}
By Theorem 2, 
\[
\Vert[T_{z}^{*},T_{z}]\Vert_{A^{2}(\Omega)\rightarrow A^{2}(\Omega)}=\Vert[T_{\varphi}^{*},T_{\varphi}]\Vert_{A^{2}(\mathbb{D})\rightarrow A^{2}(\mathbb{D})}
\]
 where $\varphi$ is the conformal map from $\mathbb{D}$ onto the
simply connected domain $\Omega$. The corollary now immediately follows
from Theorem 6.\end{proof}
\begin{rem*}
Just as the results in \cite{Olsen} and \cite{Bell} can be combined
to give a new proof of the St. Venant inequality, corollary \ref{cor:StVenantEquality}
gives another proof that equality in the St. Venant inequality characterizes
disks. 
\end{rem*}

\section{Concluding Remarks}

It must be noted that Olsen-Reguera proof only applies to Toeplitz
operators whose symbol is in the Dirichlet space, and the upper bound
is in terms of $\Vert\varphi'\Vert_{A^{2}}$ rather than the area
of $\varphi(\mathbb{D}).$ Example 4 however leads us to believe that
the upper bound in terms of the area $\varphi(\mathbb{D})$ without
multiplicity, i.e. the spectrum of $T_{\varphi}$ is, perhaps, the
correct one. If so, the result should be extendable to all Toeplitz
operators with bounded analytic symbol. One of the consequences of
the Olsen-Reguera result is a unique sharp upper bound on all extremal
problems of the type
\[
\sup_{u\in(A^{2}(\mathbb{D}))^{\perp},\|u\Vert\leq1}\vert\int_{\mathbb{D}}\bar{\varphi}udA\vert
\]
since by a standard duality argument, we have that
\[
\inf_{f\in A^{2}(\mathbb{D})}\Vert\bar{\varphi}-f\Vert_{L^{2}(\mathbb{D})}=\sup\{\vert\int_{\mathbb{D}}\bar{\varphi}udA\vert:\: u\in(A^{2}(\mathbb{D}))^{\perp},\:\|u\Vert\leq1\}
\]
(cf. \cite[§4]{Guadarrama}). We believe examining extremal problems
of this type could very well lead to a proof of the Olsen-Reguera
result which doesn't depend on power series calculations. It would
be interesting to be able to determine the sharp ``isoperimetric''
bounds for these types of extremal problems, similar to Putnam's inequality
and \cite{KhavLueking}. It would also be worthwhile to investigate
finitely connected domains. While the result of Olsen-Reguera is certainly
applicable, it would be nice to establish the deficiency based on
the number of boundary components, as in \cite{KhavLueking}. Finally,
it would be interesting to see which, if any, other so-called ``isoperimetric
sandwiches'' could be expressed in terms of $[T^{*},T]$ acting on
other function spaces, e.g. the Dirichlet space.

\end{document}